\documentclass{amsart}

\usepackage{amsmath,amssymb,amsthm,amsfonts,amscd,array,graphicx}
\usepackage{mathrsfs}
\usepackage{lipsum}

\usepackage[T2A]{fontenc}     
\usepackage[cp1251]{inputenc} 
\usepackage[english]{babel}    

\newcommand{\K}{\mathbb{K}}
\newcommand{\Z}{\mathbb{Z}}
\newcommand{\Aut}{\mathrm{Aut}}

\newcommand{\Y}{\mathrm{Y}}
\newcommand{\A}{\mathcal{A}}

\newcommand\blfootnote[1]{%
	\begingroup
	\renewcommand\thefootnote{}\footnote{#1}%
	\addtocounter{footnote}{-1}%
	\endgroup
}

\newtheorem{theorem}{Theorem}[section]

\newtheorem{lemma}[theorem]{Lemma}
\newtheorem{prop}[theorem]{Proposition}
\theoremstyle{definition}

\newtheorem{remark}[theorem]{Remark}

\author{Anton Trushin}

\title{Graded automorphisms of the algebra of polynomials in three variables}

\date{12.06.2022}

\address{Lomonosov Moscow State University, Faculty of Mechanics and Mathematics, Department of Higher Algebra, Leninskie Gory 1, Moscow, 119991 Russia}
\address{Moscow Center for Fundamental and Applied Mathematics, Moscow, Russia}
\address{HSE University, Faculty of Computer Science, Pokrovsky Boulvard 11, Moscow, 109028 Russia}


\email{Trushin.ant.nic@yandex.ru}

\begin{document}

	\maketitle
	
	\begin{abstract}
		Let us consider a polynomial algebra in three variables equipped with an integer grading. We construct a system of group-generating automorphisms that preserve a given grading.
	\end{abstract}
	
	\blfootnote{Supported by the grant RSF-DST 22-41-02019.}
	
	\section{Introduction}
	Let $\K$ be an algebraically closed field of characteristic zero and let~$\A=\K[x_1,\ldots,x_n]$ denote the polynomial ring in $n$ variables over $\K$. Let~$\varphi = (f_1,\ldots,f_n)$ is a polynomial map of the algebra $\A$.
	Automorphisms of the form $$\varphi = (x_1, \ldots, x_{i-1},\ \lambda x_i + F,\ x_{i+1}, \ldots, x_n),$$ where $F \in \K [x_1, \ldots, x_{i-1},x_{i+1}, \ldots, x_n], \lambda \neq 0$, we call {\it elementary}. Automorphisms that can be represented as a composition of elementary ones are called {\it tame}. In particular, non-degenerate linear coordinate transformations are tame automorphisms. Automorphisms that are not tame are called {\it wild}.
	
	The question of the existence of wild automorphisms is currently solved for~$n\leq 3$.
	\begin{itemize}
		\item The group of automorphisms of the algebra $\K[x]$ is easily described, namely, all automorphisms are affine transformations: $x \mapsto \lambda x + c$. This means that there are no wild automorphisms among them.
		
		\item The group of automorphisms of the algebra of polynomials in two variables K[x,y] also has no wild automorphisms~\cite{b5}, moreover, the automorphism group is an amalgamated product of two subgroups, one of which is a subgroup of affine automorphisms~\cite{b1}.
		
		\item For $n = 3$, wild automorphisms exist: it was proved in \cite{b2} that the Nagata automorphism \cite{b7} is wild.
	\end{itemize}
	Thus, non-trivial systems of generators for the groups $\Aut\left(\K[x]\right)$ and $\Aut\left(\K[x,y]\right)$ are known.
	
	Suppose now that the algebra $\A$ has a grading by some group $G$. We denote this grading by $\varGamma$. An automorphism $\varphi$ will be called {\it $\varGamma$-graded} if $\varphi(\A_g) \subset \A_g$ for any $g \in G$, where $\A_g$ is a homogeneous degree component $g$. The $\varGamma$-gra\-dy\-and\-ro\-van\-th automorphisms form a subgroup $\Aut_{\varGamma}(\A)$ in the group $\Aut(\A)$.
	A graded automorphism is called {\it graded-tame} with respect to grading~$\varGamma$ if it can be represented as a composition of elementary graded automorphisms. Otherwise, we call such an automorphism {\it graded wild}. In \cite{b3} for $n = 4$, an example of a $\Z$-grading that admits graded wild automorphisms is given, and it is shown that Anik's automorphism is graded wild with respect to this grading. \cite{b6} describes all $\Z$-gradings of the algebra $\K[x,y,z]$ that admit graded wild automorphisms. For those gradings that do not admit graded wild automorphisms, graded elementary automorphisms form a natural system of generators for the group $\Aut_{\varGamma}(\A)$. In this paper we construct systems of generators of the group $\Aut_{\varGamma}(\A)$ for $\Z$-gradings of the algebra $\A=\K[x,y,z]$ that admit wild automorphisms.
	
	\section{Graded automorphisms}
	In this section, we present the results of \cite{b6}.\\
	Consider a non-trivial $\mathbb{Z}$-grading $\varGamma$ on the algebra $\A=\K[x,y,z]$ such that the variables $x$, $y$, and $z$ are homogeneous. If there are zeros among the degrees of variables, or if the degrees of all variables have the same sign, then, as \cite{b6} proved in the paper, all graded automorphisms are graded tame. Further, we will assume that there are no zeros among the powers of the variables and that they are not all of the same sign. Multiplying all three powers simultaneously by the same number does not change the set of automorphisms. Therefore, it can be considered that
	$$(\deg_{\varGamma}(x), \deg_{\varGamma}(y), \deg_{\varGamma}(z)) = (a,b,-c),\qquad a,b,c > 0,$$
	and the greatest common divisor of $a,b$ and $c$ is equal to one. The main result of \cite{b6} is the following theorem.
	\begin{theorem}
		A graded wild automorphism of the algebra $\K[x,y,z]$ with grading $\varGamma$ exists if and only if there exists a pair of natural numbers $P$ and $Q$ such that $a = cP + bQ$ , and~$Q \geqslant 2$.
	\end{theorem}
	Consider two subgroups of the group $\Aut_\varGamma(\A)$:
	\begin{equation*}
		\mathrm{E} = \left\lbrace \varphi \in \Aut_{\varGamma}(\A) \big| \varphi(z) = z\right\rbrace,
		\qquad
		\mathrm{T}=\left\lbrace(x, y, \lambda z)\mid \lambda\neq 0\right\rbrace.
	\end{equation*}
	\begin{lemma}\cite[Lemma 3.8]{b6}\\
		The group $\Aut_{\varGamma}(\A)$ is isomorphic to the semidirect product~$\mathrm{E}\leftthreetimes\mathrm{T}$.
	\end{lemma}
	Geometrically, automorphisms of the algebra $\A$ correspond to automorphisms of the three-dimensional affine space $\mathbb{A}^3$. Consider the plane $\Y = \mathbb{V}(z - 1)$.
	The automorphisms $\mathbb{A}^3$ corresponding to the automorphisms of the group $E$ preserve the plane $\Y$. The algebra of regular functions $\K[Y]$ is isomorphic to the algebra $\K[u,v]$, where $u = x\big|_\Y;\ v = y\big|_\Y$. Thus there is a homomorphism $$\alpha: \mathrm{E} \rightarrow \Aut(\K[u,v]),$$
	and $\alpha$ is structured as follows:
	$$
	\alpha(f,g,z)=(f(u,v,1),g(u,v,1)).
	$$
	\begin{remark}\cite[Corollary 3.13]{b6}\\
		The homomorphism $\alpha$ is injective.
	\end{remark}
	
	\begin{prop}\cite[Corollary 2.2]{b6}\label{p1}\\
		If the automorphism of the plane preserves $ (0,0) $, then it decomposes into a composition of elementary ones that preserve $ (0,0) $.
	\end{prop}
	
	\begin{prop}\cite[Proposition 2.3]{b6}\label{t2}\\
		Let $\widetilde{\varGamma}$ be a grading by an abelian group of the polynomial algebra in two variables. Then all graded automorphisms of $\A=\K[x,y]$ are graded-tame with respect to $\widetilde{\varGamma}$.
	\end{prop}
	
	We say that an automorphism $\widetilde{\varphi}$ of the plane $\Y$ lifts to a space automorphism if $\alpha^{-1}(\widetilde{\varphi}) \neq \varnothing$.
	
	Consider a grading ${\widetilde{\varGamma}}$ on the algebra $\K[u,v]$ by a cyclic group $\Z_c$ such that $(\deg_{\widetilde{\varGamma}}(u), \deg_ {\widetilde{\varGamma}}(v)) = (\overline{a}, \overline{b})$, where $\overline{a}$ and $\overline{b}$ are the images of $a$ and $b$ under the natural homomorphism from $\Z$ to $\Z_c$. It is easy to see that for any $\varphi\in E$ the automorphism $\alpha(\varphi)$ is ${\widetilde{\varGamma}}$-graded. We need the following lemma from~\cite{b6}.
	\begin{lemma}\cite[Lemma 3.14]{b6}\label{ma}\\
		Let $\widetilde{\varphi} = (\widetilde{f},\widetilde{g}) \in \Aut_{\widetilde{\varGamma}}(\K[u,v])$. Then $\alpha^{-1}(\widetilde{\varphi}) = \varnothing$ if and only if $\widetilde{f}$ contains a monomial $v^q$ such that $bq < a$, or $\widetilde{g}$ has a non-zero intercept.
	\end{lemma}
	
	\begin{lemma}\cite[Lemma 3.17]{b6}\label{l3}\\
		Let $\xi_1 \in \Aut_{\widetilde{\varGamma}}(\K[u,v])$ is an elementary automorphism, $\xi_2 \in \Aut_{\widetilde{\varGamma}}(\K[u,v])$ is a linear automorphism. Then $\xi_2\circ\xi_1 = \zeta_2\circ\zeta_1\circ\zeta_0$, where $\zeta_0$ and $\zeta_2$ are linear automorphisms, $\zeta_1$ is an elementary automorphism, and $\zeta_2(u) = \lambda u$.
	\end{lemma}
	
	\section{Generator system}
	
	In the case when the grading does not admit wild automorphisms, the system of generators consists of elementary graded automorphisms. Let now the grading $\varGamma$ admit wild automorphisms.
	
	\begin{remark}\label{r2}
		If the grading $\varGamma$ admits wild automorphisms, then by Theorem 2.1
		$$\deg_{\varGamma}(f) = a > b = \deg_{\varGamma}(y).$$
	\end{remark}
	
	Let $\varphi = (f, g, z)$ and $\alpha(\varphi) = \widetilde{\varphi} = (\widetilde{f},\widetilde{g}) \in \Aut_{\widetilde{\varGamma}}(\K[u,v])$.
	
	\begin{prop}\label{p2}
		The automorphism $\widetilde{\varphi}$ can be represented as a composition of graded elementary automorphisms $\xi'_n\circ\ldots\circ \xi'_1$ such that for any $j$ the linear part of the polynomial $\xi'_j(u) )$ is equal to $\lambda_j u$ and $\xi'_j(0,0) = (0,0)$.
	\end{prop}
	
	\begin{proof}
		Due to remark \ref{r2} the automorphism $\varphi$ preserves the origin, moreover, the degree of $f$ is greater than the degree of $g$, which means that monomials of the form $yz^m$ have zero coefficients in the polynomial $f$ . This implies that the linear part of the polynomial $\widetilde{f}$ is equal to $\lambda u$. By theorem \ref{t2} and proposition \ref{p1} the automorphism $\widetilde{\varphi}$ can be decomposed into a composition of graded elementary ones: $\widetilde{\varphi} =\xi_n\circ\ldots\circ \xi_1$ such that $\xi_j(0,0) = (0,0)$. We may assume that $\xi_1$ is linear, possibly identical. We apply induction on $t$, where $0\leq t\leq n$, and show that the automorphism $\widetilde{\varphi}$ can be represented as the composition $\xi'_n\circ\ldots\circ \xi'_1$ elementary graded automorphisms such that for any $j$ such that $n-t < j\leq n$ the linear part of $\xi'_j(u)$ is equal to $\lambda_j u$.
		
		{\it Base.}\\
		For $t = 0$ the assertion is obvious.\\
		
		{\it Step.} Let us put $k = n - t$.
		By inductive hypothesis, we have $\widetilde{\varphi} =\xi_n\circ\ldots\circ \xi_1$ and for all $k < j\leq n$ linear part of $\xi_j(u)$ is equal to $\lambda_j u$ and $\xi_1$ is linear.
		
		If $\xi_{k}$ is not linear, then either the linear part of $\xi_{k}(u)$ is equal to $\lambda u$ or $\xi_{k}=(\nu u + \lambda v + f (v), v)$. In the second case, we put $\tau=(u + \lambda v, v)$. We can replace $\xi_{k}$ with $\xi_{k}'=\xi_{k}\circ \tau^{-1}$ and $\xi_{k-1}$ with $\xi_{k -1}'=\tau\circ\xi_{k-1}$. In both cases the linear part of $\xi_{k-1}'(u)$ is equal to $\lambda_j u$.
		
		If $\xi_{k}$ is linear and $k>1$, then we can assume that $\xi_{k-1}$ is not linear and $k>2$. By the lemma \ref{l3}, the composition $\xi_{k}\circ\xi_{k-1}$ can always be written as $\zeta_2\circ\zeta_1\circ\zeta_0$, where $\zeta_0$ and $\ zeta_2$ are linear, the linear part is $\zeta_2(u) = \lambda' u$, and $\zeta_1$ is elementary. Then we will replace $\xi_{k}\circ\xi_{k-1}\circ\xi_{k-2}$ with $\xi_{k}'\circ\xi_{k-1}'\circ\xi_ {k-2}'$,\ where $\xi'_{k} = \zeta_2,\ \xi'_{k-1} = \zeta_1,\ \xi'_{k-2} = (\zeta_0 \circ\xi_{k-2})$. Then the linear part of $\xi_{k}'(u)$ is equal to $\lambda'_{k} u$.
		
		Finally, if $k = 1$, then by the induction hypothesis we have $\xi'_n\circ\ldots\circ\xi'_2(u) = \lambda_2\cdot\ldots\cdot\lambda_n u$. Since the linear part of $\widetilde{\varphi}(u)$ is equal to $\lambda u$, then $\xi'_1(u) = \frac{\lambda}{\lambda_2\cdot\ldots\lambda_n}u$ .
	\end{proof}
	
	Denote the degree of the smallest monomial in a polynomial in one variable $f$ as $\underline{\deg}(f)$.
	Consider the following sets of $\widetilde{\varGamma}$-graded automorphisms of the algebra $\K[u,v]$:\\
	$$D = \{(u, \lambda v + u^k)| ka \equiv b(\mathrm{mod}\ c), \lambda \in \mathbb{K}^\times\},$$
	$$U = \{(\lambda u + f(v), v)| \deg_{\widetilde{\varGamma}}(f(v)) = \overline{a},\underline{\deg}(f) \geq \frac{a}{b}, \lambda \in \mathbb{ K}^\times\},$$
	$$W = \{(\lambda u + f(v), v)| \deg_{\widetilde{\varGamma}}(f(v)) = \overline{a}, \deg(f) < \frac{a}{b}, \lambda \in \mathbb{K}^\times \}.$$
	
	\begin{remark}
		By Lemma \ref{ma}, automorphisms from $D$ and $U$ can be lifted to space automorphisms, and automorphisms from $W$ can be lifted to space automorphisms if and only if $f = 0$.
	\end{remark}
	
	\begin{remark}
		The sets $U$ and $W$ are subgroups of the automorphism group.
	\end{remark}
	
	Let $\tau \in W$, $\theta \in D$, let $\tau^{-1}\circ\theta\circ\tau(u,v) = (\widetilde{f}(u,v ),\widetilde{g}(u,v))$.\\
	
	\begin{remark}\label{r5}
		The linear part of the polynomial $g$ is equal to $\lambda v, \lambda \neq 0$, moreover, the automorphism preserves the origin, that is, $(\widetilde{f}(0,0),\widetilde{g}(0, 0)) = (0,0)$.
	\end{remark}
	
	\begin{lemma}\label{up}
		For elementary automorphisms $\tau \in W$, $\theta \in D$, there exists an automorphism $s_{\tau,\theta} \in W$ such that the automorphism $s_{\tau,\theta}\circ\ tau^{-1}\circ\theta\circ\tau$ lifts to a space automorphism.
	\end{lemma}
	
	\begin{proof}
		By the remark \ref{r5}, an automorphism $\tau^{-1}\circ\theta\circ\tau$ does not lift to a space automorphism if and only if the polynomial $\widetilde{f}(u,v) $ contains monomials of the form $\nu v^m$, where $m < \dfrac{a}{b}, \nu \neq 0$. Among all such monomials, we choose a monomial with the minimum degree $m_1$. Now let $\lambda$ be the coefficient of the monomial $v$ in the polynomial $\widetilde{g}$. The coefficient $\lambda$ is nonzero due to the remark \ref{r5}. Consider the automorphism $s_1 = (u - \dfrac{\nu}{\lambda}v^{m_1},v)$. Consider a new automorphism $s_1\circ\tau^{-1}\circ\theta\circ\tau$ and the image $\widetilde{f}_1(u,v)$ of the variable $u$ under this automorphism. The minimum degree of monomials of the form $v^m$ in this automorphism has increased. So $s_{\tau,\theta} = s_l\circ\ldots\circ s_1$.
	\end{proof}
	
	Let us introduce the following notation: $\tau_\theta = s_{\tau,\theta}\circ\tau^{-1}; S = \{\tau_\theta\circ\theta\circ\tau|\tau \in W, \theta \in D\}.$
	
	\begin{lemma}
		The group $\widetilde{E} = \alpha(E)$ is generated by the subgroup $U$ and the set $S$.
	\end{lemma}
	
	\begin{proof}
		Take $\widetilde{\varphi} \in \widetilde{E}$.
		Let $\widetilde{\varphi} = \xi_n\circ\ldots\circ\xi_1$, where $\xi_k$ are elementary graded automorphisms, and, according to the proposition \ref{p2}, we can assume that $\xi_k$ lies in $D\cup U\cup W$. Let among these automorphisms $m$ automorphisms from $D$. We show by induction on $m$ that the automorphism $\widetilde{\varphi}$ can be represented as a composition of automorphisms from $U$ and $S$.\\
		
		{\it Base.}\\
		For $m = 0$ the automorphism $\widetilde{\varphi}$ lies in $\langle W\cup U\rangle$. Moreover, it rises to a space automorphism, and hence lies in $U$.
		
		{\it Step.}\\
		Let $\xi_{k-1},\ldots,\xi_1$ lie in $U\cup W$ and $\xi_k$ lie in $D$. Let $\xi_{k-1}\circ\ldots\circ\xi_1 = \tau\circ\tau_1$, where $\tau\in W, \tau_1\in U$. Let us rewrite the automorphism $\widetilde{\varphi}$ as $\widetilde{\varphi} =\xi_n\circ\ldots\circ\xi_{k+1}\circ\tau\circ\tau^{-1}\circ \xi_k\circ\tau\circ\tau_1$. By the lemma \ref{up}, there is an automorphism $s_{\tau,\xi_{k}} \in W$ such that the automorphism $s_{\tau,\xi_{k}}\circ\tau^{-1} \circ\xi_k\circ\tau = \tau_{\xi_k}\circ\xi_k\circ\tau$ lifts to a space automorphism. Thus $$\widetilde{\varphi} = \left(\xi_n\circ\ldots\circ\xi_{k+1}\circ\tau_1\circ s^{-1}\right)\circ\left(\ tau_{\xi_k}\circ\xi_k\circ\tau\right)\circ\tau_1.$$ Moreover, the composition $\xi_n\circ\ldots\circ\xi_{k+1}\circ\tau\circ s^{ -1}$ rises to a space automorphism and the inductive assumption holds for it.\\
	\end{proof}
	
	Thus, the following theorem is proved.
	
	\begin{theorem}
		Automorphisms of the algebra $\Aut_{\varGamma}(\A)$ are generated by automorphisms from $\alpha^{-1}(U)$, automorphisms of the form $\alpha^{-1}(\tau_{\theta}\circ\ theta\circ\tau)$, where $\tau \in W, \theta \in D$ and automorphisms from the group $\mathrm{T}=\left\lbrace(x, y, \lambda z)\mid \lambda \neq 0\right\rbrace$.
	\end{theorem}
	
	\section{Examples}
	Let the algebra $\K[x,y,z]$ be graded as follows: $(\deg(x),\deg(y),\deg(z)) = (3,1,-1)$.
	Such a grading admits graded wild automorphisms according to Theorem 1. Under such a grading we have:
	$$D = \{(u, \lambda v + \mu u^k)|ka \equiv b(\mathrm{mod}\ c), k > 0, \lambda \in \mathbb{K}^\times \},$$
	$$U = \{(\lambda u + f(v), v)|\underline{\deg}(f) \geq 3; \lambda, \mu \in \mathbb{K}^\times\},$$
	$$W = \{(\lambda u + f(v), v)|\underline{\deg}(f) > 1, \deg(f) < 3; \lambda,\mu \in \mathbb{K}^\times\}.$$
	
	Let $\tau = \begin{pmatrix}
		u + v^2\\
		v
	\end{pmatrix} \in W$ and let $\theta = \begin{pmatrix}
		u\\
		v+u
	\end{pmatrix} \in D$. Then we have:
	$$\tau^{-1}\circ\theta\circ\tau = \begin{pmatrix}
		u - u^2 - v^4 - 2uv - 2v^3 - 2uv^2\\
		v + u + v^2
	\end{pmatrix}$$
	Now consider the corresponding automorphism of the polynomial algebra in three variables:
	$$\sigma = \alpha^{-1}(\tau^{-1}\circ\theta\circ\tau) = \begin{pmatrix}
		x - x^2z^3 - y^4z - 2xyz - 2y^3 - 2xy^2z^2\\
		y + xz^2 + y^2z\\
		z
	\end{pmatrix}$$
	The resulting automorphism $\sigma$ coincides with the Nagata automorphism \cite{b7}. Thus, the Nagata automorphism lies in the system of generators of the group of graded automorphisms.\\
	Let us present the general form of an automorphism of the form $\tau_\theta\circ\theta\circ\tau$.\\ Let
	$\tau = \begin{pmatrix}
		\lambda_1u + \nu v^2\\
		v
	\end{pmatrix} \in W$ and let $\theta = \begin{pmatrix}
		u\\
		\lambda_2 v +\mu u^k
	\end{pmatrix} \in D$.
	Then
	
	$$\tau^{-1}\circ\theta\circ\tau = \begin{pmatrix}
		u + \frac{\nu}{\lambda_1}v^2 - \frac{\nu}{\lambda_1} \left(\lambda_2v + \mu(\lambda_1 u + \nu v^2)^k\right) ^2\\
		\lambda_2 v + \mu(\lambda_1 u + \nu v^2)^k
	\end{pmatrix}$$
	
	The coefficient of $v^2$ in the polynomial $\tau^{-1}\circ\theta\circ\tau(u)$ is equal to $\dfrac{(1-\lambda_1^2)\nu}{\lambda_1}$ .\\ Thus, the automorphism $\tau_\theta\circ\theta\circ\tau$ is equal to $s_{\tau,\theta}\circ\tau^{-1}\circ\theta\circ\tau$ , where $$s_{\tau,\theta} = \begin{pmatrix}
		u - \frac{(1-\lambda_1^2)\nu}{\lambda_1\lambda_2^2}v^2\\
		v
	\end{pmatrix}.$$

\end{document}